\documentclass[12pt,reqno,sumlimits]{amsart}

 \usepackage{amssymb,amscd,amsmath,epsfig,mathtools}

\usepackage{geometry}                		
\usepackage{graphicx}				

\usepackage[utf8]{inputenc}
\usepackage[english]{babel}
\usepackage{cite}

\usepackage{amssymb}
\usepackage{amsmath}
\usepackage{tikz-cd}
\usepackage{amsthm}
\newtheorem{theorem}{Theorem}[section]
\newtheorem{definition}{Definition}[section]
\newtheorem{corollary}{Corollary}[section]
\newtheorem{lemma}{Lemma}[section]
\newtheorem{proposition}{Proposition}[section]
\theoremstyle{definition}

\newtheorem{remarks}{Remarks}[section]
\newtheorem{example}{Example}[section]

\newcommand{\Lip}{{\rm Lip}}
\newcommand{\isd}{\imath_d}
\newcommand{\dte}{d_\eta}
\newcommand{\isde}{\imath_{\dte}}

\theoremstyle{definition}

\newcommand{\R}{{\mathbb R}}

\newcommand{\calC}{{\mathcal C}}

\newcommand{\calZ}{{\mathcal Z}}
\renewcommand{\to}{\longrightarrow}

\newsavebox{\savepar}

\numberwithin{equation}{section}

\newcounter{labelflag} 
\newcommand{\labelon}{\setcounter{labelflag}{1}}
\newcommand{\Label}[1]{
                       \ifnum\thelabelflag=1
                          \ifmmode
                             \makebox[0in][l]{\qquad\fbox{\rm#1}}
                          \else
                             \marginpar{\vspace{0.7\baselineskip}
                                        \hspace{-1.1\textwidth}
                                        \fbox{\rm#1}}
                          \fi
                       \fi
                       \label{#1}
                      }
\labelon

\begin{document}

\title{Central Limit Theorems on Compact Metric Spaces}
\author{Steve Rosenberg}
\address{Department of Mathematics and Statistics\\Boston University}
\email{sr@math.bu.edu}
\author{Jie Xu}
\address{Department of Mathematics and Statistics\\Boston University}
\email{xujie@math.bu.edu}
\maketitle						


 \section{Introduction}
We produce a series of Central Limit Theorems (CLTs) associated to compact metric measure spaces $(K,d,\eta)$, with $\eta$ a reasonable probability measure.   For the first CLT, we can ignore $\eta$
by isometrically embedding $K$ into $\calC(K)$, the space of
continuous functions on $K$ with the sup norm,  and then applying known CLTs for sample means on Banach spaces (Theorem \ref{thm:3.1}).  However, the sample mean makes no sense back on $K$, so using $\eta$ we develop a CLT for the sample Fr\'echet mean (Corollary \ref{cor:3.1}).  This involves working on the closed convex hull of the embedded image of $K$. 
To work in the easier Hilbert space setting of $L^2(K,\eta)$, we have to modify the metric $d$ to a related metric
$\dte$.  We  obtain an $L^2$-CLT for both the sample mean and the sample Fr\'echet mean (Theorem
\ref{thm:5.1}), and we relate the Fr\'echet sample and population means on the closed convex hull to the Fr\'echet means on the image of $K$.
Since the $L^2$ and $L^\infty$ norms play important roles, in \S6 we develop a metric-measure
criterion relating $d$ and  $\eta$ under which 
all $L^p$ norms are equivalent.

\section{Background Material}
Throughout the paper, $(K,d)$ will be a compact metric space.   Recall that a $ G $-valued random variable $ X $ is a function
$
X : \Omega \rightarrow G$,
where $ (\Omega, \mathcal{F}, P) $ is a fixed probability triple. The induced measure/distribution on $ G $ is given by $ X_{*}(P) $, with
\begin{equation*}
X_{*}(P)(A) = P(X^{-1}(A)), \ A \subset G.
\end{equation*}

We recall the setup and statement of a Central Limit Theorem on Banach spaces due to Zinn.

\begin{definition} \label{defone}(i) Let $G$ be a Banach space.  A probability measure $\gamma$ on $G$ is  a Gaussian Radon 
measure  if for every nontrivial linear functional $L:G\to \R$, on $ G $, the pushforward measure $ L_{*}(\gamma) $ is a non-degenerate Gaussian measure on $ \mathbb{R} $, i.e., a standard Gaussian measure with non-zero variance.

(ii) Let $ X_{1}, \dotso, X_{n}, \dotso $ be  any set of $ G $-valued i.i.d. random variables with common distribution $ \mu $. $ \mu $ satisfies the G-Central Limit Theorem (G-CLT) on  $ G $ if there exists a Guassian Radon probability measure $ \gamma $ on $ G $ such that the distributions, $ \mu_{n} $, of $ \frac{X_{1} + \dotso + X_{n}}{\sqrt{n}} $ converge, i.e., for every bounded -continuous real-valued function $ f $ on $ G $,
\begin{equation*}
\int_{G} f d\mu_{n} \rightarrow \int_{G} f d\gamma.
\end{equation*}

(iii)  The metric $d$ on $K$ implies Gaussian continuity (or $d$ is CGI)   if whenever $ \lbrace X(s) \rbrace_{s \in K} $ is a separable Gaussian process such that
\begin{equation*}
{\mathbb E}\left[\lvert X(t) - X(s) \rvert^{2}\right]\leqslant d^{2}(t, s),
\end{equation*}
\noindent then $ X $ has continuous sample paths a.s.. 
\end{definition}

Let $ H_{d}(K, \epsilon) = \log (N_{d}(K, \epsilon)) $, where  $ N_{d}(K, \epsilon) $ is  the smallest number of 
$d$-balls of diameter at most $ 2 \epsilon $ which cover $ K $.

\begin{proposition}\cite[Thm.~3.1]{dudley} \label{proptwoone} If 
\begin{equation}\label{eq2.1}
\int_{0}^{\infty} H_{d}^{1/2} (K, u) du < \infty,
\end{equation}
\noindent then $ d$ is GCI.
\end{proposition}

 Let $ \mathcal{C}(K) $ be the Banach space of continuous functions on the compact metric space $ (K, d) $ equipped with the sup-norm $ \lVert \cdot \rVert_{\infty} $. 
 $\calC(K)$ becomes a complete metric space with the induced  distance function $ d_{\infty} $ by $ d_{\infty}(f, g) = \lVert f - g \rVert_{\infty}, \forall f, g \in \mathcal{C}(K) $.

\begin{definition} For the compact metric space $ (K, d) $, set \begin{equation*}
\Lip(d) = \lbrace x \in \mathcal{C}(K) : 
 \sup_{t \neq s} \frac{\lvert x(t) - x(s) \rvert}{d(t, s)} < \infty \rbrace.
\end{equation*}
\end{definition}

 $ \Lip(d) $ is nonempty (by letting $x$ be a constant function). We check that $\Lip(d)$ is closed. If  $ \lbrace x_{k} \rbrace \in \Lip(d)$ has $ \lim_{k \rightarrow \infty} x_{k} = y \in \mathcal{C}(K) $, then for any 
 $\epsilon >0$ and $t\neq s$,
\begin{align*}\frac{\lvert y(t) - y(s) \rvert}{d(t, s)} &\leq \frac{\lvert y(t) - x_j(t) \rvert}{d(t, s)}
+\frac{\lvert x_j(t) - x_j(s) \rvert}{d(t, s)} + \frac{\lvert x_j(s) - y(s) \rvert}{d(t, s)}\\
&\leq 2\epsilon + \frac{\lvert x_j(t) - x_j(s) \rvert}{d(t, s)},
\end{align*}
for $j = j(\epsilon) \gg 0$ independent of $t,s.$  This implies that $y\in \Lip(d).$

\begin{definition} A Radon probability measure $ \mu $ on the Banach space $ (G, \lVert \cdot \rVert) $ has zero mean and finite variance, respectively, if
\begin{equation}\label{conds}
 \int_{G} x\ \mu(dx) = 0, \ 
 \int_{G} \lVert x \rVert^{2} \ \mu(dx) < \infty,
\end{equation}
respectively.
\end{definition}

Of course, if a sequence of $G$-valued random variables $X_i$ has finite expectation, then 
the new random variable $ X_i - \mathbb{E}[X_i] $ has zero mean.  

We can now state Zinn's CLT.

\begin{theorem}\cite{zinn} \label{twoone}  Let $ (K, d) $ be a compact metric space with $d$ CGI. If  $ \mu $ is a Radon probability measure on $ Lip(d) $ with zero mean and finite variance, then $ \mu $ satisfies the Central Limit Theorem on $ (\mathcal{C}(K), \Vert\cdot\Vert_\infty) $ in the sense of Definition \ref{defone}(ii). 
\end{theorem}

For our main results, we need to define the  Fr\'echet mean
with respect to a probability measure $Q$ on $(K,d)$.  This generalizes 
the notion of centroids from vector spaces to metric spaces.

\begin{definition} (i) The Fr\'echet function $ f:K\to \R$ is
\begin{equation*}
f(p) = \int_{M} d^{2}(p, z) Q(z) dz, p \in M.
\end{equation*}
If $ f(p) $ has a unique minimizer $ \mu = {\rm argmin}_{p \in K} f(p) $, we call $\mu$ the Fr\'echet mean of $ Q $. 

(ii) 
Given an i.i.d. sequence $ X_{1}, \dotso, X_{n} \sim Q $ on $ M $, the empirical Fr\'echet mean is defined to be
\begin{equation*}
\mu_{n} = {\rm argmin}_{p \in M} \frac{1}{n} \sum_{i=1}^{n} d^{2}(p, X_{i}),
\end{equation*}
provided the argmin is unique.\end{definition}


Unlike centroids in Euclidean space, the uniqueness of Fr\'echet mean cannot be guaranteed, even in 
spaces which locally look like Euclidean space.

\begin{example} 
We parametrize an open cone (minus a line) $\calZ: x^2+y^2 = z^2$ of height one by
$$F(u,v) = \left(\frac{1}{\sqrt{2}} u\cos v, \frac{1}{\sqrt{2}} u\sin v, \frac{1}{\sqrt{2}} u\right),\  (u,v)\in (0,1)\times
(0,2\pi).$$
There is a Riemannian  isometry from the sector $S = \{(r,\theta)\in (0, \sqrt{2})\times (0, \sqrt{2}\pi)\}$ to $\calZ$ induced by
$\alpha: (r,\theta)\mapsto (u = r/\sqrt{2}, v = \sqrt{2}\theta)$, {\it i.e.,} 
$$(r,\theta)\mapsto 
 \left(\frac{r}{2}\cos( \sqrt{2}\theta), \frac{r}{2}\sin( \sqrt{2}\theta),\frac{r}{2}\right).$$
 Indeed, the first fundamental form of the sector at $(r,\theta)$, respectively the first fundamental form of the cone at $(u,v)$, are
 $$\begin{pmatrix}1&0\\0&r^2\end{pmatrix},\ \ \begin{pmatrix}2&0\\0&\frac{u^2}{2}\end{pmatrix},$$
 respectively.  It is easy to check that the differential $d\alpha$ preserves these inner products.  Thus every point $p\in S$ has a neighborhood $U$ such that the usual Euclidean distance between 
 $q_1, q_2\in U$  equals the geodesic distance between $\alpha(q_1), \alpha(q_2).$

It is easy to check that for e.g. the uniform distribution on $S$, the Fr\'echet mean/cen-troid $(\bar x, \bar y)$  is inside $S$.  In contrast, by the rotational symmetry of the geodesic distance function on $\calZ$, the minima of the Fr\'echet function on $\calZ$ form a circle containing $\alpha(\bar x, \bar y).$

\end{example}

For results on CLTs when the Fr\'echet mean is not unique, see \cite{Bhattacharya}.

\section{A CLT for  Compact Metric Spaces}

 In this section, we isometrically embed the compact metric $( K,d) $ into the Banach space
 $(\calC(K),d_\infty)$  to obtain a CLT on the image of $ K $.
 
 We define 
$$
\imath_{d} : K \rightarrow \mathcal{C}(K),  x \mapsto f_{x} : = d(x, \cdot).$$

The following proposition is well known.
\begin{proposition} $ \imath_{d} : (K, d) \rightarrow (\imath_{d}(K), d_{\infty}) $ is an isometry.  
\end{proposition}
\begin{proof} For  $ x, y \in K $, we have
\begin{equation*}
d_{\infty}(f_{x}, f_{y}) = \lVert f_{x} - f_{y} \rVert_{\infty} = \max_{z \in K} \lvert d(x, z) - d(y, z) \rvert \geqslant \lvert d(x, y) - d(y, y) \rvert = d(x, y).
\end{equation*}
\noindent On the other hand, for  $ x, y,z \in K $, we have
$$
 d(x, z) - d(y, z) \leqslant d(x, y), d(y, z) - d(x, z) \leqslant d(x, y) \Rightarrow \lvert d(x, z) - d(y, z) \rvert \leqslant d(x, y),$$
 so
 $$
 \max_{z \in K} \lvert d(x, z) - d(y, z) \rvert \leqslant d(x, y) \Rightarrow d_{\infty}(f_{x}, f_{y}) \leqslant d(x, y).$$
Thus $\imath_d$ is an isometry.
\end{proof}

It follows that $\imath_d$ is an injection, and $\imath_d(K)$ is a compact subset of $\calC(K)$.

To obtain a CLT on $ \imath_{d}(K) $ from Theorem~\ref{twoone}, we need to verify its hypotheses.

\begin{lemma} \label{lem:one}
$ \imath_{d}(K) \subset \Lip(d) $.
\end{lemma}
\begin{proof} For $ f_{x} \in \imath_{d}(K) $,  the triangle inequality for $s,t\in K$ gives
\begin{equation*}
\lvert f_{x}(t) - f_{x}(s) \rvert = \lvert d(x, t) - d(x, s) \vert \leqslant d(s, t).
\end{equation*}
\noindent It follows that
\begin{equation*}
\sup_{s \neq t} \frac{\lvert f_{x}(s) - f_{x}(t) \rvert}{d(s, t)} \leqslant  1.
\end{equation*}

\end{proof}

In the following proof, we strongly use the fact that $\calC(K)$ is a ``linearization" of $K$.

\begin{lemma} \label{lem:two} The metric  $ d $ on $K$ is GCI. 
\end{lemma}

\begin{proof} We must verify (\ref{eq2.1})
in Proposition~\ref{proptwoone}.  Equivalently, we will show 
\begin{equation}\label{eq3.2} \int_{0}^{\infty} H_{d_{\infty}}^{\frac{1}{2}}(\imath_{d}(K), u) du < \infty.
\end{equation} 

As a compact set, $ \imath_{d}(K) $ can be covered by $ N $ balls of radius $ 1 $ for some $ N $.
Fix any point $ x_{0} \in \imath_{d}(K) $, and consider the $ d_{\infty} $-ball $ B_\infty(x_{0}, 1) $ of radius $ 1 $ centered at $ x_{0} $. 
The closure $\overline{B_\infty(x_0,1)}$ equals 
$\imath_d(\overline{B_d(\imath^{-1}_d(x_0),1) }  )$ of the corresponding ball in $K.$
It follows that $\overline{B_\infty(x_0,1)}$ is compact, so
 we can cover $ B(x_{0}, 1) $ by $ M $ balls of radius $ \frac{1}{2} $ for some $ M $. 
 
Since 
 $ d_{\infty} $ is translation invariant, $ M $ is independent of the choice of $ x_{0} $. 
Moreover, $d_\infty$ is homogeneous in the sense that $ d_{\infty}(cf, cg) = \lvert c \rvert d_{\infty}(f, g) $ for  $ c \in \mathbb{R} $. Thus for $ r >0 $, any $ d_\infty $-ball of radius $ r $ contained in $ \imath_{d}(K) $ can be covered by $ M $ balls of radius $ \frac{r}{2} $. Hence
\begin{equation*}
N_{d_{\infty}}(\imath_{d}(K), 2^{-k}) \leqslant N \cdot M^{k + 1}.
\end{equation*}

 To estimate (\ref{eq2.1}), we  integrate over $ [0,1] $ and $ [1, \infty) $ separately.
Since $ \imath_{d}(K) $ is compact, it is covered by a single $ d_\infty $-ball $ B_\infty(x_{0}, R) $ for some $ R \gg 0 $ and a fixed $ x_{0} \in \imath_{d}(K) $. Choose $ k_{0} \in \mathbb{N} $ such that $ 2^{k_{0}} \leqslant R < 2^{k_{0} + 1} $. We have
\begin{align*}
\int_{1}^{\infty} H_{d_{\infty}}^{\frac{1}{2}} (\imath_{d}(K), u) du & = \int_{1}^{\infty} \sqrt{\log (N_{d_{\infty}}(\imath_{d}(K), u))} du \\
&\leqslant \sum_{k=1}^{\infty} \sqrt{\log (N_{d_{\infty}}(\imath_{d}(K), 2^{k}))} (2^{k} - 2^{k - 1}) \\
& \leqslant \sum_{k = 1}^{k_{0} + 1} \sqrt{\log (N_{d_{\infty}}(\imath_{d}(K), 2^{k}))} 2^{k - 1} < \infty.
\end{align*}
\noindent For the region $ [0, 1] $, we have
\begin{align*}
\lefteqn{
\int_{0}^{1} H_{d_{\infty}}^{\frac{1}{2}} (\imath_{d}(K), u) du}\\ 
& = \int_{0}^{1} \sqrt{\log (N_{d_{\infty}}(\imath_{d}(K), u))} du
 \leqslant \sum_{k=0}^{\infty} \sqrt{\log (N_{d_{\infty}}(\imath_{d}(K), 2^{-k-1}))} (2^{-k} - 2^{-k - 1}) \\
&\leqslant \sum_{k = 0}^{\infty} \sqrt{ \log (N \cdot M^{k + 2})} \ 2^{-k - 1}
 = \sum_{k = 0}^{\infty} \sqrt{ (k + 2) \log M + \log N} \ 2^{ - k - 1} 
 < \infty.
\end{align*}
Adding these estimates gives (\ref{eq3.2}).
\end{proof}

This gives our first Central Limit Theorem on $K $, or really on the isometric space $\isd(K).$

\begin{theorem} \label{thm:3.1}Let $ (K, d) $ be a compact metric space, let $ \mu $ be a Radon probability measure on $K$ with finite variance and such that $\imath_{d,*}\mu$ has zero mean on
$\isd(K)$. Then $\imath_{d,*}\mu$ satisfies the $G$-CLT for $G = (\calC(K), \Vert\cdot\Vert_\infty).$
\end{theorem}
\begin{proof} By Lemmas \ref{lem:one}, \ref{lem:two}, the hypotheses of Theorem 2.1 are satisfied
for $\imath_{d,*}\mu$.
\end{proof}


\section{A  Fr\'echet CLT associated to a compact metric space}
In the previous section, we found a  G-CLT on Banach space associated with the usual sample mean $ \sqrt{n} \cdot \frac{1}{n} \sum_{i=1}^{n} X_{i} $ on $G = \calC(K)$. In this section,  we prove a G-CLT on the compact metric space $K$ (again, really on $\isd(K)$), endowed with a Radon probability measure $\eta,$  for the sample Fr\'echet mean
\begin{equation}\label{sfmean}
\text{argmin}_{Y \in \imath_{d}(K)} \frac{1}{n} \sum_{i=1}^{n} \lVert X_{i} - Y \rVert_{2,\eta}^{2},
\end{equation}
 $ X_{1}, \dotso, X_{n} $ are i.i.d.  $ \imath_{d}(K) $-valued random variables, and the $L^2$ norm is taken with respect to $\imath_{d,*}\eta.$

Note that we compute the sample Fr\'echet mean with respect to the $ L^{2}$-norm, since we will need a Hilbert space structure below. 
The minimizer of (\ref{sfmean}) may not exist or be unique, since $\imath_{d}(K) $ may be neither closed nor  in
 $ \mathcal{C}(K) $. Instead, we consider the closed convex hull of $ \imath_{d}(K) $,
 on which the uniqueness of the Fr\'echet mean is guaranteed.
 
\begin{definition} \label{cch} Let $ \imath_{d}(K)^{c} $ be the  convex hull of $ \imath_{d}(K) $, {\it i.e.}, the
intersection of all convex subsets of $ \mathcal{C}(K) $ containing $ \imath_{d}(K) $, and let
\begin{equation*}
S_{d} = S_d(K) = \overline{  \imath_{d}(K)^{c}}
\end{equation*}
\noindent be the closure of $ \imath_{d}(K)^c $.
\end{definition}
By  \cite[Thm.~5.35]{border}, $ S_{d} $ is a compact subset of $ \mathcal{C}(K) $.  
It is easy to check that $S_d(K) \subset \Lip(d).$
As the minimizer of a convex function on a closed convex space, the Fr\'echet mean is unique.

\begin{example} To continue with  Example 2.1,  choose a probability measure $Q$ on the cone $\calZ$. The sample Fr\'echet mean for $K$-valued random variables $Y_i$ lies in the interior of $\calZ$ in $\R^3.$  It is unclear if the sample Fr\'echet mean for the $\imath_d(\calZ)$-valued random variables $X_i = \isd\circ Y_i$ lies in
$\isd(\calZ)$, but it certainly lies in $S_d(\calZ).$  (This example doesn't really show the strength of embedding $\calZ$ into $\calC(\calZ)$, since $\calZ$ lies in a linear space.

Similar remarks apply to the Fr\'echet minimum. While the Fr\'echet minimum for the cone $(\calZ,\eta)$ is not unique, the Fr\'echet minimum on 
$(S_d(\calZ), \imath_{d,*}\eta)$, the closed convex hull of the isometric set $(\imath_d(\calZ), \imath_{d,*}\eta)$, is unique.  (Note that the sector $S$ in Example 2.1 is only locally isometric to 
$\calZ.$) While we have gained uniqueness, there is no reason why the Fr\'echet minimum need be inside $\imath_d(\calZ)$, as in Example 2.1.  It is shown in \cite[Supplement C]{eric} that in general the distance from the Fr\'echet mean in $S_d(K)$ to $\imath_d(K)$ is at most twice the diameter of $K$.

\end{example}

At this point we have the embeddings $K\hookrightarrow \isd(K) \subset S_{d}(K) \subset \mathcal{C}(K) \subset L^{2}(K)$, where $L^2(K)$ is taken with respect to a probability measure on $K$.  While 
$K\hookrightarrow L^2(K)$ is no longer an isometry, there is a known CLT on $S_d(K)$ equipped with the $L^2$ norm.

\begin{theorem}\label{thm:4.1} 
Let $\mu$ be a 
  Radon probability measure  supported in $ K $ such that $\imath_{d,*}\mu$ satisfies
 (\ref{conds}).  Then 
$\imath_{d,*}\mu$ satisfies the $G$-CLT for $G = (L^2(K), \Vert\cdot\Vert_{2,\eta})$. 
The same result holds if the random variables 
 $\{X_i\}$ in the G-CLT are $\isd(K)$-valued and/or $\mu$ has support in $S_d(K).$
 \end{theorem}

\begin{proof} By  \cite[Thm.~9.10]{ledoux}, the Hilbert space $L^2(K)$ is of type 2 and cotype 2. The existence
of a CLT on spaces of type/cotype 2 follows from \cite[Thm.~3.5]{Jorgensen}.
\end{proof}

We also obtain a  CLT for the sample Fr\'echet  mean. Here the Hilbert space structure works to our
advantage, as the sample Fr\'echet mean and the usual sample mean coincide..

\begin{proposition} \label{prop:4.1}For $S_d(K)$-valued random variables $\{X_i\}$, we have
\begin{equation*}
S_{n} := \frac{1}{n} \sum_{i=1}^{n} X_{i} = {\rm argmin}_{Y \in S_{d}} \frac{1}{n} \sum_{i=1}^{n} \lVert X_{i} - Y \rVert_{2,\eta}^{2}.
\end{equation*}

\end{proposition}

\begin{proof} 
It is well-known that in a finite dimensional Euclidean space, the sample mean coincides with
 the sample Fr\'echet mean. For fixed $x\in K$, the  real-valued  random variables $ \lbrace X_{i}(x) \rbrace $ 
 satisfy
 \begin{equation*}
\frac{1}{n} \sum_{i=1}^{n} \lvert X_{i}(x) - S_{n}(x) \rvert^{2} \leqslant \frac{1}{n} \sum_{i=1}^{n} \lvert X_{i}(x) - Y(x) \rvert^{2},\  \forall Y \in \mathcal{C}(K).
\end{equation*}
Therefore, for all $Y$,
$$
 \frac{1}{n} \sum_{i=1}^{n} \int_{K} \lvert X_{i}(x) - S_{n}(x) \rvert^{2} dQ(x) \leqslant \frac{1}{n} \sum_{i=1}^{n} \int_{K} \lvert X_{i}(x) - Y(x) \rvert^{2} dQ(x),$$
 so
$$ \frac{1}{n} \sum_{i=1}^{n} \lVert X_{i} - S_{n} \rVert_{2,\eta}^{2} \leqslant \frac{1}{n} \sum_{i=1}^{n} \lVert X_{i} - Y \rVert_{2,\eta}^{2}, $$
which implies  
$$S_{n} = \text{argmin}_{Y \in S_{d}} \frac{1}{n} \sum_{i=1}^{n} \lVert X_{i} - Y \rVert_{2,\eta}^{2}.$$

\end{proof}

Combining the Proposition with Theorem \ref{thm:4.1} gives us a CLT for the sample Fr\'echet mean.  We
set $\Vert\cdot\Vert_{2,\eta}$ be the $L^2$ norm with respect to a measure $\mu$, and set $\calC_0(X)$ to be the set of bounded continuous functions on a topological space $X$. 

\begin{corollary}\label{cor:3.1}
(i)  Let $\{X_i\}$ be i.i.d. $ S_{d} $-valued random variables with distribution $\mu$ a
  Radon probability measure  supported in $ \imath_{d}(K) $ satisfying 
 (\ref{conds}).  
 Then there exists a Gaussian Radon probability measure 
 $ \tilde{\gamma}_{2} $ such that the distributions $ \mu_{n} $ of 
 $${\rm argmin}_{Y \in S_{d}} \frac{1}{\sqrt{n}} \sum_{i=1}^{n} \lVert X_{i} - Y \rVert_{2,\eta}^{2} $$ 
 converge weakly to $\tilde\gamma_2$ in the sense of Defintion \ref{defone}.
 
 (ii) $\gamma_2$ in Theorem \ref{thm:4.1} equals $\tilde\gamma_2$ in distribution. In particular, for
 $f\in \calC_0(S_d(K))$, 
\begin{equation*}
 \int_{S_d(K)} f d\gamma_{2} = \int_{S_d(K)} f d\tilde{\gamma}_{2}.
\end{equation*}
\end{corollary}

(iii) Let $\gamma_1$ be the limiting measure obtained in Theorem \ref{thm:3.1}.  Then $\gamma_1 = \gamma_2$. 

\begin{proof} (i) follows from Proposition \ref{prop:4.1}.  For (ii), the Proposition implies that  the distributions
 of the sample mean and the sample Fr\'echet mean are the same a.s.  For (iii), this seems to be because Theorems \ref{thm:3.1} and \ref{thm:4.1} are essentially the same.
 \end{proof}


\section{$L^2$ techniques and G-CLTs}

In this section, we embed the compact metric space $K$,  now 
equipped with a Radon measure $\eta$ and a modified metric, into $L^2(K,\eta)$ to produce an $L^2$ version of a  G-CLT.  In this Hilbert space setting, we are able to relate the Fr\'echet means of the closed convex hulls to the Fr\'echet means on the embedded image of $K$.

We  define a seminorm on $ \imath_{d}(K) = \{f_x = d(x,\cdot) : x\in K\}$ by
\begin{equation*}
\lVert f_{x} \rVert^2_{2, \eta} =\int_{\imath_{d}(K)} \lvert f_{x}(y) \rvert^{2} d\imath_{d, *} \eta(y)= \int_{K} d^{2}(x, y) d \eta(y).
\end{equation*}
Taking the completion the space of norm zero  functions gives the Hilbert space $ (L^{2}(\isd(K)), \lVert \cdot \rVert_{2, \eta}) $. More precisely, we will prove $L^2(\isd(K))$-CLTs for both the sample mean and the sample Fr\'echet mean.

The norm $\Vert\cdot\Vert_{2,\eta}$ induces a metric $d_{2,\eta}$ on $\isd(K)$.  Since $\isd:(K,d)\to(\isd(K),
d_{2,\eta})$ is easily continuous, we can pull back $d_{2,\eta}$ to a metric $d_\eta := \imath_{d}^{*} d_{2, \eta}$
 on $K$:

\begin{equation*}
 d_{ \eta}(x, y) = d_{2, \eta}(f_{x}, f_{y}) = \left(\int_K (d(x,z) - d(y,z))^2 d\eta(z)\right)^{1/2} .
\end{equation*}

Thus $\imath_{\dte}:(K, \dte)\to (\imath_{\dte}(K), d_{2,\eta})$, defined by $\imath_{\dte}(x) = \dte(x,\cdot)$, 
is an isometry.  We interpret $(K, \dte)$ as a modification of $(K,d)$ which keeps track of the $L^2$ information
of $\eta.$

We want to relate the various metrics. Let $d_\infty$ be the metric on $\calC(K)$ induced by the sup norm $\Vert\cdot\Vert_\infty,$ and let $[\isd(K)]$ be the image of $\isd(K)$ in $L^2(K).$  Consider the maps
$$(K, d) \stackrel{\isd}{\to}(\isd(K), d_\infty)\stackrel{F}{\to} ([\isd(K)], d_{2,\eta}) \stackrel{G}{\to} (K, d_\eta)$$
given by $F(f_x) = [f_x]$, where we take the $L^2$ equivalence class, and $G([f_x]) = \imath_{\dte}^{-1}(f_x)
= \imath_{\dte}^{-1}\isd(x).$
(We show that $G$ is well-defined below.)  $\isd$ is an isometry.

In general, $ F $ and $ G $ are not injective, since  equivalence classes in $ L^{2}(K) $ have many
representatives, without a restriction on $\eta.$ 

\begin{lemma} \label{lem:5.1} Assume that   every $d$-ball $ B_{\epsilon}(x)$ centered at  $x \in K$
has $ \eta(B_{\epsilon}(x)) > 0 $.
  Then $F$ is injective, and $G$ is well-defined and injective.
\end{lemma}

Since $F$ and $G$ are trivially surjective, they are bijective under the Lemma's hypothesis.  Note that 
for Lebesgue measure and the standard metric on $\R^N,$ the hypothesis is satisfied, while delta functions give rise to Radon measures that do not satisfy the hypothesis.  

\begin{proof} 
For $F$, it suffices to show that $F\circ \isd:x\mapsto [f_x]$ is injective.
For $ x \neq y $, 
and $\epsilon < d(x,y)/3$, we have
$$d(y,z) \geq d(x,y) - d(x,z)> 3\epsilon -\epsilon > d(x,z)+ \epsilon,$$
for all $z\in B_\epsilon(x).$
Therefore \begin{align*}
d_{2, \eta}([f_{x}], [f_{y}])^{2} &= \int_{K} \lvert f_{x}(z) - f_{y}(z) \rvert^{2} d\eta_{z} \geq \int_{B_{\epsilon}(x)} \lvert d(x, z) - d(y, z) \rvert^{2} d \eta_{z}\\
& > \epsilon^{2} \eta(B_{\epsilon}(x)) > 0.
\end{align*}
Thus $[f_x]\neq [f_y].$

To show that $G$ is well-defined, if $ f_{x}, f_y\in  [f_{x}] \in L^{2}(K) $, then
\begin{equation*}
d_{2, \eta}([f_{x}], [f_{y}]) = 0 \Rightarrow \int_{K} \lvert d(x, z) - d(y, z) \rvert^{2} d\eta(z) = 0.
\end{equation*}
As above, this implies that $x=y$, so $[f_x]$ has a unique representative of the form $f_x.$  Since $\isd,\imath_{\eta}$ are injective, it follows that $G$ is injective.

\end{proof}

We can now state and prove an $L^2$ CLT on  $ S_{\dte}(K) $ for both the sample mean and the sample Fr\'echet mean.  As before, let  $ S_{\dte} = S_{\dte}(K) $ be the closed convex hall of $ \imath_{\dte}(K) $ in the sup norm.

\begin{theorem} \label{thm:5.1}
Let $\{X_i\}$ be i.i.d. $ S_{\dte}$-valued random variables with distribution  $\mu$ a
  Radon probability measure  supported in $S_{\dte}(K) $ 
  satisfying  (\ref{conds}).  Assume that the hypothesis of Lemma \ref{lem:5.1} holds.
 
(i)  There exists a Gaussian Radon probability measure $ \gamma $ on $ S_{\dte} $ such that the distributions $ \mu_{n} $, of $ \frac{X_{1} + \dotso + X_{n}}{\sqrt{n}} $ converge to a probability measure
 $ \gamma $ in the sense of Definition \ref{defone}(ii).  
 
 (ii) The distributions $ \tilde{\mu}_{n} $ of $ 
 {\rm argmin}_{Y \in S_{\dte}} \frac{1}{\sqrt{n}} \sum_{i=1}^{n} \lVert X_{i} - Y \rVert_{2,\eta}^{2} $ 
 converge in the same
 sense  to the same  measure $ \gamma $.

 (iii) Under the hypothesis in Lemma \ref{lem:5.1}, the distributions $ \tilde{\mu}_{n} $ of\\
  $ 
 {\rm argmin}_{Y \in S_{\dte}} \frac{1}{\sqrt{n}} \sum_{i=1}^{n} \lVert X_{i} - Y \rVert_{\infty}^{2} $ 
 converge in the same
 sense  to a Gaussian Radon probability measure $ \tilde{\gamma} $.

\end{theorem}

\begin{proof} 
(i) Replacing the metric $d$ by $\dte$ in Theorem \ref{thm:3.1} gives the CLT for the sample mean.  

(ii) Applying Corollary \ref{cor:3.1}(i) and (iii) to $\dte$ gives the convergence of $\tilde\mu_n$ to the same measure $\gamma.$ 

(iii) 
The main idea is to use the isometric bijection $\imath_{\dte}\circ \isd^{-1}: (\isd(K), \Vert\cdot\Vert_\infty)\to (\imath_{\dte}(K), 
\Vert\cdot\Vert|_{2,\eta}).$  This extends linearly to an isometric bijection
$$\imath_{\dte}\circ \isd^{-1}: (S_d(K), \Vert\cdot\Vert_\infty)\to (S_{\dte}(K), 
\Vert\cdot\Vert|_{2,\eta}).$$
Set
\begin{equation*}
Z_{n} := \text{argmin}_{Y \in S_{\dte} }\frac{1}{n} \sum_{i=1}^{n} \lVert X_{i} - Y \rVert_{\infty}^{2}.
\end{equation*}
By Proposition \ref{prop:4.1},
\begin{align*}
Z_{n} & = 
 \text{argmin}_{Y \in S_{d_{2, \eta}}} \frac{1}{n} \sum_{i=1}^{n} \lVert \imath_{d} \circ \imath_{\dte}^{-1} X_{i} - \imath_{d} \circ \imath_{\dte}^{-1} Y \rVert_{2, \eta}^{2} \\
& = 
\imath_{\dte}\circ\isd^{-1}\left(
\text{argmin}_{Y \in S_d} \frac{1}{n} \sum_{i=1}^{n} \lVert \imath_{d} \circ \imath_{\dte}^{-1} X_{i} - Y \rVert_{2, \eta}^{2}\right)\\
 &= 
 \imath_{\dte}\circ\isd^{-1}\left(
 \frac{1}{n} \sum_{i=1}^{n} \imath_{d} \circ \imath_{\dte}^{-1} X_{i}\right).
\end{align*}
$\{(\imath_{d} \circ \imath_{\dte}^{-1}) (X_{i})\}$ are $S_d$-valued i.i.d. random variables with common distribution 
$(\imath_{d} \circ \imath_{\dte}^{-1})_*\mu.$  By Theorem \ref{thm:4.1}, we obtain a 
$S_d$-CLT with respect to a Gaussian Radon measure 
$\gamma'$ on $S_d$.  The isometry $\imath_{\dte}\circ \isd^{-1}$ then gives the $S_{\dte}$-CLT with respect to
$(\imath_{\dte}\circ \isd^{-1})_*\gamma'.$
\end{proof}


Because we are in a Hilbert space setting, we can prove that the Fr\'echet sample and population means on 
$S_{\dte}$ and on $\isde(K)$ have simple relationships.

Let  $S_{\dte}$ be the closed convex hull of $\imath_{\dte}(K) := K_0$ in $L^2(K, \eta),$ and let $f^2_y = d_\eta(y, \cdot) \in \calC(K).$   Let 
\begin{align*}\bar F(\bar x) &= 
\int_{S_{\dte}} d^2_{L^2}(\bar x,\bar y)d\imath_{\dte,*}\eta(\bar y) 
= \int_{K_0} d^2_{L^2}(\bar x,\bar y)d\imath_{\dte,*}\eta(\bar y) 
 = 
 \int_K d^2_{L^2}(\bar x, f^2_y) d\eta(y),\\
 F(x) &= 
  \int_K d^2_{\eta}(x,y) d\eta(y) = 
 \int_{K_0} d^2_{L^2}(f^2_x, f^2_y) )d\imath_{\dte,*}\eta(f^2_y),
 \end{align*}
 be the $L^2$ Fr\'echet functions  of $S_{\dte}$ and $K$, respectively, and let 
 $$\bar \mu = {\rm argmin}_{\bar x\in S_{\dte}} \bar F(\bar x),\ \mu = {\rm argmin}_{ x\in K} \int_K d^2_\eta(x,y) d\eta(y)$$
 be the population means on $S_{\dte}$ and $K$, respectively.  Set $\mu_0 = \isde(\mu).$
 
 We note that  as the minimum of a convex function on a convex set, $\bar\mu$ is unique.  Also, gradients of differentiable functions exist in Hilbert spaces.  

\begin{proposition}\label{prop:5.1}  Assume that (i) $\mu$ is unique, (ii) $K_0$
is the zero set of a Fr\'echet differentiable function $H:L^2(K,\eta)\to \R$ with
$\nabla H_{\mu_0} \neq 0.$   Then
$\mu_0$ is the closest point in $K_0$ to $\bar\mu.$  The same relationship holds for the sample 
Fr\'echet means of $K_0$-valued i.i.d.~random variables.
\end{proposition}  

\begin{proof}  
Note that $K_0 = \partial K_0$ in $L^2(K, \eta)$, since a compact subset of an infinite dimensional space has no interior.  Also,  $\bar\mu \in K_0$ implies $\bar\mu = \mu_0$, so we may assume $
\bar\mu \not\in K_0.$

The method of Lagrangian multipliers is valid in $L^2(K,\eta)$, so there exists $\lambda\in\R$ with
$\nabla \bar F_{\mu_0} = \lambda \nabla H_{\mu_0}.$  The differential $D\bar F$ at $\mu_0$  is given by
\begin{align*} 
D\bar F_{\mu_0}(v) &= (d/dt)|_{t=0}  \int_{S_{\dte}} d^2_{L^2}(\mu_0 + tv, f^2_y) d\eta(y)\\
& =
(d/dt)|_{t=0}  \int_{K}
  \langle \mu_0 + tv- f^2_y,\mu_0 + tv- f^2_y\rangle d\eta(y)\\
& = 2\left\langle v,
\mu_0 - \int_K  f^2_y d\eta(y)\right\rangle,
\end{align*}
where the last term equals the Hilbert space integral 
$$\int_{K_0}\bar y d\imath_{\dte,*}\eta(\bar y) = \int_{S_{\dte}}\bar y d\imath_{\dte,*}\eta(\bar y).$$
Thus $\nabla \bar F_{\mu_0} =2( \mu_0 - \int_{S_{\dte}}\bar y d\imath_{\dte,*}\eta(\bar y)).$  Since 
$\nabla \bar F_p = 0$ only at $p = \bar\mu$, we see that $\bar \mu = \int_{S_{\dte}}\bar y d\imath_{\dte,*}\eta(\bar y).$ (This is the usual statement that the Fr\'echet mean is the center of mass of a convex set in
$\R^n.$)  Thus $\nabla \bar F_{\mu_0} =2( \mu_0 - \bar\mu).$  

Since $ \mu_0 - \bar\mu$ is a multiple of $\nabla H_{\mu_0}$, which is perpendicular to the level set $\isde(K)$,
we have $\mu_0-\bar\mu \perp \partial K_0$.  We have not used that $\mu_0$ is a minimum, so the same perpendicularity holds at any critical point of 
$\bar F$ on $\partial K_0.$  We can translate in $S_{\dte}$ so that $\bar \mu = 0$, in which case
$\nabla \bar F_p = 2p$ is twice the Euler vector field.  The level sets $\bar F^{-1}(r)$ are thus spheres centered at the origin in $S_{\dte}.$  Since $\mu_0$ is on the lowest level set of any point in $K_0$, $\mu_0$ is closer to 
the origin than any other critical point of $\bar F$ on $\partial K_0.$ 

If we consider the distance function $D:K_0\to \R, D(\bar x) = d^2(\mu_0, x)$, then a Lagrangian multiplier argument as above shows that at a critical point $p$ of $D$, we have $\mu_0-p \perp \partial K_0.$  Thus 
$\mu_0$ is a critical point of $D$, and by the last paragraph $\mu_0$ must be the closest point in $K_0$ to 
$\mu_0.$

The same argument holds for the sample means.
\end{proof}

If a closest point  $p(z)\in  K_0$ to each 
$z\in S_{\dte}$ can be chosen so that $p$ is continuous, as in the unlikely case that $K_0$ is convex, then we  get a G-CLT on $K_0$ and hence on $K$ for $p_*\tilde \mu_n, p_*\gamma.$  This would 
connect Theorem \ref{thm:5.1} and Proposition \ref{prop:5.1}.

\section{Relating $L^p$ norms}

We have results for $L^\infty$- and $L^2$-CLTs, so we wish to compare the associated metrics on $\isd(K)$. 
Since $\isd(K)$ is compact, all norms are abstractly equivalent.
 In this section, we introduce a metric-measure assumption under which the $L^p$ norms on $\isd(K)$ are 
explicitly equivalent, i.e., the constants in the norm comparisons are explicit.

Of course, for a probability measure  $\eta$ on $K$, we have for $f\in\calC(K)$ and $p\in [1,\infty)$
$$\Vert f \Vert_p^p = \int_K |f|^p d\eta \leq \Vert f\Vert _\infty^p \int_K d\eta = \Vert f\Vert_\infty^p.$$
Thus 
\begin{equation}\label{eq:6.1} d_p(f,g) \leq d_\infty(f,g),
\end{equation}
 for $f, g\in \calC(K)$ and $d_p, d_\infty$, respectively, the $L^p, L^\infty$ metrics, respectively.  

We would like a reverse inequality on $\isd(K)$ in (\ref{eq:6.1}), which is impossible without any assumptions.  In particular, we assume that $f_x = d(x,\cdot) \in \isd(K)$ is $\eta$-measurable for $x\in K.$
We also strengthen the triangle inequality
$$|d(x,z) - d(y,z) | \leq d(x,y)$$
as follows:
\medskip

\noindent {\bf Assumption:}
There exist $C, D\in (0,1)$ such that for all $x,y\in K$ 
\begin{equation}\label{eq:6.2}\eta\{z\in K: |d(x,z) - d(y,z)| \leq D\cdot d(x,y)\} < C.
\end{equation}

\medskip
The intuition is that
for $D =0$, a workable $\eta$ has $\eta\{z\in K: |d(x,z) - d(y,z)| \leq D\cdot d(x,y)\} = 0$, as 
for Lebesgue measure. For
$D$ close to zero, we demand that 
$$\max_{x,y\in K} \eta\{z\in K: |d(x,z) - d(y,z)| \leq D\cdot d(x,y)\}$$
should be strictly less than $1 = \eta(K)$.  The assumption fails for delta function measures on $\R^n$, but appears to hold for normalized Legesgue measure on compact subsets of $\R^n.$

\begin{proposition}  Assume that $f_x$ is $\eta$-measurable for $x\in K$ and 
that (\ref{eq:6.2}) holds. For $p\in [1,\infty), p'\in [1,\infty]$ with $p < p'$,
we have
$$D(1-C)^{1/p}\cdot d_{p'}( f_x,f_y) \leq  d_p( f_x,f_y) \leq d_{p'}( f_x,f_y),$$
for all $x, y\in K.$  Thus the $L^p$ norms are explicitly equivalent.
\end{proposition}

\begin{proof}
The fact that $d_p( f_x,f_y) = \Vert f_x-f_y\Vert_p\leq \Vert f_x-f_y\Vert_{p'} =d_{p'}( f_x,f_y)$ 
is contained in the H\"older inequality proof that $L^{p'}\subset L^{p}$ on a finite measure space.

Set $S_{x,y} = \{z\in K: |d(x,z) - d(y,z)| \leq D\cdot d(x,y)\}.$
Using (\ref{eq:6.2}), we have 
\begin{align*}  \Vert f_x - f_y\Vert_p^p &\geq \int_{K\setminus S_{x,y}} |d(x,z) - d(y,z)|^p d\eta(z)
\geq \int_{K\setminus S_{x,y}} D^p\cdot d^p(x,y)\  d\eta(z)\\
&\geq D^p(1-C) d^p(x,y) = D^p(1-C)\Vert f_x- f_y\Vert^p_\infty.
\end{align*}
Thus
$$d_p(f_x, f_y) \geq D(1-C)^{1/p}d_\infty(f_x, f_y)\geq  D(1-C)^{1/p} d_{p'}(f_x, f_y),$$
by (\ref{eq:6.1}).  
\end{proof}

\begin{remarks}
(i) The $d_p$ pull back to metrics on $K$, also denoted $d_p$, which by the Proposition are all explicitly equivalent.

(ii) For fixed $x,y$, set $K^p_D = 
\{z\in K: |d_p(x,z) - d_p(y,z)| \leq D\cdot d_p(x,y)\}.$  Thus $K^\infty_D$ is the set in (\ref{eq:6.2}).  From the estimates in the proof above, we obtain
\begin{align*} 
\Vert f_x-f_z\Vert_{p'} - \Vert f_y-f_z\Vert_{p'}  &\leq D^{-1}(1-C)^{-1/p}\Vert f_x-f_z\Vert_p - 
\Vert f_y-f_z\Vert_p,\\
D\Vert f_x-f_y\Vert_p &\leq D\Vert f_x-f_y\Vert_{p'}.
\end{align*} 
Thus
$$\{z: D^{-1}(1-C)^{-1/p}\Vert f_x-f_z\Vert_p -\Vert f_y-f_z\Vert_ p
\leq  D\Vert f_x-f_y\Vert_p\} \subset K^{p'}_D,$$
and the same with $x, y$ switched.
Since 
 $D^{-1}(1-C)^{-1/p} >1$, we obtain
\begin{align*}\MoveEqLeft[4]
  \{z:d_p(x,z) - d_p( y,z) \leq D d_p(x,y)\} \\ 
   &\subset
\{z: D^{-1}(1-C)^{-1/p}\Vert f_x-f_z\Vert_p -\Vert f_y-f_z\Vert_ p
\leq  D\Vert f_x-f_y\Vert_p\} \subset K^{p'}_D,
\end{align*}
using $d_p(x,y) = d_p(f_x, f_y).$  The same holds with $x, y$ switched, so
$K^p_D \subset K^{p'}_D$ for $p, p'$ as in the Proposition.  Thus (\ref{eq:6.2}) has the greatest chance of holding for $d = d_1.$
\end{remarks}

\bibliographystyle{plain}
\bibliography{CLT}

\end{document}